\documentclass[a4paper,11pt]{article}
\usepackage{amssymb}
\usepackage{enumerate,verbatim}
\usepackage{amsmath}
\usepackage{amsfonts}
\usepackage{amsthm}
\usepackage{xcolor}
\usepackage{epsfig}
\usepackage[normalem]{ulem}
\usepackage[utf8]{inputenc} 
\usepackage{nicefrac} 

\newtheorem{theorem}{Theorem}
\newtheorem{remark}{Remark}
\newtheorem{definition}{Definition}
\newtheorem{example}{Example}

\newtheorem{lemma}{Lemma}
\newtheorem{proposition}{Proposition}

\newcommand{\C}{\mathbb{C}}

\newcommand{\Z}{\mathbb{Z}}

\def\eqref#1{(\ref{#1})}

\begin{document}
\date{}
\author{
Colin Christopher\\
School of Engineering, Computing and Mathematics\\Plymouth University\\
Plymouth, PL4 8AA, UK\\
\tt{C.Christopher@plymouth.ac.uk}\\
{}\\
Chara Pantazi \\
Departament de Matem\`atiques\\ Universitat Polit\`ecnica de Catalunya (EPSEB)\\
Av. Doctor Maranon, 44-50, 08028, Barcelona, Spain\\
\tt{chara.pantazi@upc.edu}\\
{}\\
Sebastian Walcher\\
Fachgruppe Mathematik, RWTH Aachen\\
52056 Aachen, Germany\\
\tt{walcher@mathga.rwth-aachen.de}
}
\title{{Liouvillian integrability\\ of rational vector fields: \\The case of algebraic extensions}}
\maketitle

\begin{abstract}
As shown in a previous paper, whenever a rational vector field on $\mathbb C^n$, $n>2$, is Liouvillian integrable, then it admits a first integral obtained by two successive integrations from a one-form with coefficients in a finite algebraic extension $L$ of the rational function field $K$. In the present work we discuss and characterize exceptional vector fields in this class, for which -- by definition -- the choice $L=K$ is not possible. In particular we show that exceptional vector field exist, giving explicit constructions in dimension three.\\
 MSC (2020): 34A99, 34M15, 34M50, 12H05, 12F10.\\
  Key words: Differential form, Liouvillian function, group representation, Picard-Vessiot extension, Platonic solids.
\end{abstract}

\section{Introduction}

 This paper continues our study \cite{ACPW} of Liouvillian integrability, of $n$-dimensional rational vector fields over $\mathbb C$ (in other words, of rational  $(n-1)$-forms in $n$ variables). In \cite{ACPW} it was proven that Liouvillian integrability always implies the existence of a first integral that is obtained by two successive integrations from a one-form with coefficients in a finite algebraic extension $L$ of the rational function field $K=\mathbb C(x_1,\ldots,x_n)$. This result extends a classical theorem due to Singer \cite{Singer1992} for $n=2$, where one may take $L=K$. But in \cite{ACPW} we left open the question whether there actually exist exceptional cases, i.e., Liouvillian integrable vector fields (in dimension greater than two) for which necessarily $L\not=K$. The goal of the present paper is to answer this question. Clearly one may assume here that $L$ is Galois over $K$ (see also \cite{ACPW}).\\
After some preliminaries, and noting some classes of regular (i.e. non-exceptional) rational vector fields on $\mathbb C^n$, we proceed to determine properties of exceptional vector fields. We start from a distinction whether (i) there exists a first integral obtained by integrating a closed rational one-form, or  (ii) no such integral exists. We then show that exceptional cases may be subdivided in two types. The first type (appearing in case (ii) only) is built from Galois extensions of degree less than $n$, with the Galois group acting as a permutation group on the integrals. (In dimension three, one thus has quadratic Galois extensions.) For the second type the Galois group admits a faithful representation in $GL(m,\mathbb C)$, with some $m$, $1<m<n$. (In particular, when $n=3$, then the group has a faithful representation in $GL(2,\mathbb C)$.)\\
In the final section we construct exceptional rational vector fields in dimension three, using Picard-Vessiot extensions with dihedral Galois groups, resp.\ with Galois groups that are isomorphic to symmetry groups of the Platonic solids.

\section{Background and statement of the problem}

We will refer frequently to notions and results from \cite{ACPW}. In particular we will freely use differential forms over the rational function field  $K:=\mathbb C(x_1,\ldots,x_n)$ and its extension fields. For any extension $L$ we denote by $L'$ the space of one-forms over $L$. Note that $L'$ is a vector space over $L$, of dimension $n-1$.

For preparation and as a reminder, we collect some notions and facts from \cite{ACPW}:

\begin{itemize}
\item A differential extension field $L$ of $K=\mathbb C(x_1,\ldots,x_n)$, is {\it Liouvillian} if and only if $K$ and $L$ have the same constants
 and there exists a tower of fields of the form
\begin{equation}\label{LiouvExeq} K=K_0 \subset K_1 \subset \ldots \subset K_m=L, \end{equation}
such that for each $i\in\{0,\ldots,m-1\}$ we have one of the following (see \cite{ACPW}, Definition 1 and Remark 1):
\begin{list}{}{}
\item{(i)} $K_{i+1}=K_i(t_i)$, where $t_i\not=0$ and $dt_i=\delta_i t_i$ with some $\delta_i \in K'_i$ (necessarily $d\delta_i=0$).
\item{(ii)} $K_{i+1}=K_i(t_i)$, where $dt_i=\delta_i$ with $\delta_i \in K'_i$  (necessarily $d\delta_i=0$).
\item{(iii)} $K_{i+1}$ is a finite algebraic extension of $K_i$.
\end{list}
\item As  noted in \cite{ACPW}, Remark 1, the condition on the constants is unproblematic. In particular, any finite algebraic extension of $K$, with a natural extension of differentials, is Liouvillian.
\end{itemize}

We consider an $n$-dimensional rational vector field
\begin{equation}\label{vfieldnd}
\mathcal{X}=\sum_{i=1}^nP_i\frac{\partial }{\partial x_i}
\end{equation}
on $\mathbb C^n$, $n\geq2$; equivalently an ($n-1$)-form
\begin{equation}\label{2form}
\Omega=\sum_{i=1}^nP_i\,dx_1\wedge\cdots\wedge\widehat{dx_i}\wedge\cdots dx_n
\end{equation}
defined over $K=\mathbb C(x_1,\ldots,x_n)$. 

\begin{definition}\label{dLf1}\textup{
A non-constant
element, $\phi$, of a Liouvillian extension of $K$ is called a {\it Liouvillian first integral}  of the vector field $\mathcal{X}$ if it satisfies $\mathcal{X} \phi=0$ or, equivalently, $d\phi \wedge  \Omega=0$.}
\end{definition}
\begin{remark}\label{defcondrem}\textup{ By \cite{ACPW}, Remark 4,  according to Definition \ref{dLf1}, a Liouvillian first integral  exists if and only if there exist some Liouvillian extension $L$ of $K$ and one-forms $\omega\in L'\setminus\{0\}$, $\alpha\in L'$ such that 
\begin{equation}\label{probsetup}
    \omega\wedge\Omega=0,\quad d\omega=\alpha\wedge\omega, \quad d\alpha=0.
\end{equation}
Given this situation, we will briefly say that the {\it 2--form $\Omega$ is Liouvillian integrable} and, slightly abusing language, we will say specifically that {\em $\Omega$ is Liouvillian integrable over $L$}.  }
\end{remark}

 In order to properly frame the problem, we introduce more notation.
\begin{definition}\label{workdef}
    Given a differential field extension $L$ of $K$, denote by $\mathcal I =\mathcal{I}(\Omega)$ the set of all $0\not=\omega\in L'$ for which there exists $\alpha\in L'$ so that
\eqref{probsetup}
   is satisfied.
\end{definition}
Assume that we have $\mathcal I\not=\emptyset$ for given $\Omega$ and some Liouvilian extension $L$ of $K$. We wish to  characterize {the} ``smallest possible'' extensions $\widetilde L\subseteq L$  of $K$ such that $\mathcal I\cap \widetilde L'\not=\emptyset$. 

By Singer's classical result \cite{Singer1992} for two dimensional vector fields, one has $\widetilde L=K$ in the case $n=2$. For $n>2$, a  principal result of our previous work (\cite{ACPW}, Theorem 4) may be stated as follows.

\begin{theorem}\label{algaddendum}
Let $K = \C(x_1,\ldots,x_n)$, and let $\Omega$ be the $(n-1)$-form {\eqref{2form}} over $K$. If there exists a Liouvillian first integral of $\Omega$, then there exists a first integral that is defined over a finite algebraic extension $L$ of $K$; thus 
there exist 1--forms $\omega,\, \alpha \in L'$ such that \eqref{probsetup} holds.
\end{theorem}
The present paper is concerned with these
\begin{center}
{\bf Fundamental Questions}:
Let a $(n-1)$-form $\Omega$ be given.
\end{center}
\begin{enumerate}
\item If there exists a Liouvillian extension $L\supseteq K$  with  ${\mathcal I}\not=\emptyset$, is ${\mathcal I}\cap K'\not=\emptyset$?
\item If ${\mathcal I}\cap K'=\emptyset$, characterize (minimal) algebraic extensions $\widetilde L$ of $K$ such that $\widetilde L'\cap \mathcal I\not=\emptyset$.
\end{enumerate}
\begin{remark}\label{fieldirrel}{\em 
    For given Liouvillian integrable $\Omega$ the answer to the first  fundamental question does not depend on the choice of the extension $L$. For instance, the answer being positive for some extension $L$ means that there exist $\omega,\,\alpha \in K'$ for which \eqref{probsetup} holds. But then, $\omega,\alpha\in  M'$ for any extension $M$ of $K$, so the answer is positive for $M$, too.}
\end{remark}
Whenever $\mathcal I\cap K'\not=\emptyset$, then we speak of a {\em regular case}; otherwise we speak of an {\em exceptional case}.\\
For $n=3$, i.e., for two-forms over the rational function field in three variables we obtained criteria for regular cases in \cite{ACPW}, but we gave no proof that exceptional cases actually exist. In the present paper we will provide a characterization of exceptional cases, and show that exceptional vector fields exist in dimension three.


\section{Preliminaries}
In the following discussion of the fundamental question we may and will assume that $L$ is a finite Galois extension of $K$, with group $G=Gal(L:K)$. Moreover we let $\omega\in \mathcal I$, such that \eqref{probsetup} holds with some $\alpha\in L'$.

\begin{lemma}\label{intermedfieldlem}
    Let $K\subseteq F\subseteq L$, with $\omega\in {\mathcal I}\cap F' $.
    Then there exists $\widehat\alpha\in F'$ such that 
    \begin{equation*}
        d\omega =\widehat \alpha\wedge \omega,\quad d\widehat \alpha=0.
    \end{equation*}
\end{lemma}
\begin{proof}
    The extension $L\supseteq F$ is Galois; call its Galois group $H$. There is $\alpha\in L'$ satisfying condition \eqref{probsetup} in Definition \ref{workdef}, and consequently 
    \begin{equation*}
        d\omega =\sigma(\alpha)\wedge \omega,\quad d\sigma(\alpha)=0
    \end{equation*}
    for every $\sigma\in H$. Averaging over all elements of $H$ shows the assertion.
\end{proof}
In a familiar case fundamental question is known to have a positive answer. 
\begin{lemma}\label{ratfi}
    Let $\omega\in {\mathcal I}$ be exact, thus $\omega=d\ell$ with some $\ell\in L$. Then $\Omega$ admits a rational first integral.
\end{lemma}
\begin{proof}
Take any nonconstant coefficient of the minimal polynomial of $\ell$ over $K$.
\end{proof}
{
There is a further class of field extensions that is reasonably well understood with regard to our fundamental question.
\begin{lemma}\label{purex}
Let $K\subseteq F\subseteq L$, and assume that the Galois group of $L$ over $F$ is cyclic. If $\omega\in L'$ satisfies condition \eqref{probsetup} with some $\alpha\in F'$, then ${\mathcal I}\cap F'\not=\emptyset$.\\
In particular when $G$ is cyclic, and condition \eqref{probsetup} holds with some $\alpha\in K'$, then the fundamental question has a positive answer.
\end{lemma}
\begin{proof}
 Denote the degree of $L$ over $F$ by $m$. By Lang \cite{SergeL}, Ch.~VI, Thm.~6.2 there exists $g\in F$ such that the polynomial $T^m-g\in F[T]$ is irreducible, and $L$ is the splitting field of this polynomial.
Thus let $t\in L$ such that $t^m=g$. Then $L=F[t]$ and $1,t,\ldots,t^{m-1}$ form a basis of $L$ over $F$. Moreover with $mt^{m-1}\,dt=dg$ and $t^{m-1}=g/t$ one finds $dt=t\cdot\frac{dg}{m\cdot g}$.\\
Now 
\begin{equation*}
    \omega=\eta_0+t\eta_1+\cdots +t^{m-1}\eta_{m-1}\text{  with all }\eta_i\in F',
\end{equation*}
and $\omega\wedge\Omega=0$ implies that all $\eta_i\wedge\Omega=0$. Likewise, $\alpha\in F'$ implies that $\alpha\wedge\omega=\sum t^i(\alpha\wedge\eta_i)$, and with 
\begin{equation*}
    d(t^i\eta_i)=t^i\left(\frac{i}{m}\frac{dg}{g}\wedge\eta_i+d\eta_i\right)
\end{equation*}
the condition $d\omega=\alpha\wedge\omega$ implies
\begin{equation*}
    d\eta_i=\left(\alpha-\frac{i}{m}\frac{dg}{g}\right)\wedge\eta_i,\quad 0\leq i<m.
\end{equation*}
Since some $\eta_j\not=0$, we are done.
\end{proof}
}

Due to the following observation from \cite{ACPW}, we may modify $\omega$ by a factor in $L^*$.
\begin{lemma}\label{omegamultilem}
    With $0\not=\ell\in L$, $\omega/\ell$ is a first integral of $\Omega$, and the following variant of \eqref{probsetup} holds:
    \begin{equation}\label{Liouv1}
d\left(\frac{\omega}{\ell}\right)=\left(\alpha-\frac{d \ell}{\ell}\right) \wedge \frac{\omega}{\ell},\quad d\left(\alpha-\frac{d \ell}{\ell}\right) =0.
\end{equation}
\end{lemma}
Thus, with no loss of generality we may assume
    \begin{equation}\label{blankx1}
          \omega=dx_1+\cdots.
    \end{equation}
    We define
\begin{equation}\label{g0def}
    G_0:=\left\{\sigma\in G;\, \sigma(\omega)\wedge\omega=0\right\}.
\end{equation}
\begin{lemma}\label{G0simplem}
Given $\omega$ as in \eqref{blankx1},
    one has $\sigma(\omega)\wedge\omega=0$ for some $\sigma\in G$ if and only if $\sigma(\omega)=\omega$.
Thus $G_0$ is a subgroup, $\sigma(\omega)=\omega$ if and only if  $\sigma\in G_0$, and the images of $\omega$ under the action of $G$ stand in 1-1 correspondence with the cosets $G/G_0$.
\end{lemma}
\begin{proof}
We have, in detail,
\[
   \omega=dx_1+\sum_{i\geq 2} c_i\, dx_i
\]
with $c_i\in L$; thus 
\[
 \sigma(\omega)=dx_1+\sum_{i\geq 2} \sigma(c_i)\,dx_i.
\]
Now
\[
0=\sigma(\omega)\wedge\omega=\sum_{i\geq 2}\left(c_i-\sigma(c_i)\right)\,dx_1\wedge dx_i+\sum_{1<k<i}a_{ki}\,dx_k\wedge dx_i
\]
(with suitable $a_{ik}\in L$) shows that all $\sigma(c_i)=c_i$, whence $\sigma(\omega)=\omega$.
\end{proof}
We take care of a particular case right away. The following was shown by different arguments in \cite{ACPW}, Cor.\ 1, for $n=3$.
\begin{proposition}\label{boreprop}
If $\Omega$ admits only one independent Liouvillian first integral, then $\mathcal I\cap K'\not=\emptyset$.
\end{proposition}
\begin{proof}
It suffices to consider the algebraic case; with the $L$-vector space spanned by $\mathcal I$ of dimension one. Then we have $\sigma(\omega)\wedge\omega=0$ for all $\sigma\in G$. Moreover we may choose $\omega$ as in Lemma \ref{G0simplem}, hence  $\sigma(\omega)=\omega$ for all $\sigma\in G$, which implies $\omega\in K'$.
\end{proof}

From here on we let \[\omega_1=\omega,\omega_2,\ldots,\omega_N\] be the distinct images under $G$ of $\omega$, with $d\omega_i=\alpha_i\wedge\omega_i$, $1\leq i\leq N$. We set
    \[
    V:=\left< \omega_1,\ldots,\omega_N\right>_L.
    \]
    Moreover we define
    \begin{equation}
        \widetilde V:=\left\{\beta\in L';\, \beta\wedge\Omega=0\right\}
    \end{equation}
      noting that $V\subseteq\widetilde V$ and $\dim_L\,\widetilde V=n-1$. By Proposition \ref{boreprop} we may restrict attention to the case $\dim_L\,V>1$.

\section{The structure of $\mathcal I$}
 We first note an auxiliary lemma:
    \begin{lemma}\label{auxlem} 
        Given $\omega_{i_1},\ldots,\omega_{i_e}$ such that $\omega_{i_1}\wedge\cdots\wedge\omega_{i_e}\not=0$, there exist $\eta_1,\ldots,\eta_{n-1-e}\in K'$ such that all $\eta_j\wedge\Omega=0$ and
        \begin{equation}
            \omega_{i_1}\wedge\cdots\wedge\omega_{i_e}\wedge \eta_1\wedge\cdots\wedge\eta_{n-1-e}=\ell\cdot \Omega
        \end{equation}
        for some nonzero $\ell\in L$.
    \end{lemma}
    \begin{proof}
 Extend $\omega_{i_1},\ldots,\omega_{i_e}$ to a basis of $\widetilde V$, noting that the extension may be carried out with elements of $K'$.
    \end{proof}
    \begin{remark} Consider the special case that $\dim_L\,V=n-1$, with basis $\omega=\omega_1,\ldots,\omega_{n-1}$. With representatives  ${\rm id}=\tau_1,\ldots,\tau_{n-1}$ of $G/G_0$ one then has $\omega_j=\tau_j(\omega_1)$, $1\leq j\leq n-1$, and 
    \[
    d\omega_j=\tau_j(d\omega_1)=\tau_j(\alpha_1)\wedge\omega_j=\alpha_j\wedge\omega_j,\text{  with  } \alpha_j=\tau_j(\alpha_1),\quad d\alpha_j=0
    \]
    for $ 1\leq j\leq n-1$.
    Thus $\omega_1\wedge\ldots\wedge\omega_{n-1}=\ell\cdot\Omega$ for some $\ell\in L$,
    and taking differentials one sees
    \[
    d\Omega=\widehat\beta\wedge\Omega; \quad \widehat\beta=\left(\sum_j\alpha_j-\dfrac{d\ell}{\ell}\right),\quad d\widehat\beta=0.
    \]
    Averaging $\beta=\dfrac{1}{|G|}\sum_{\sigma\in G}\sigma(\widehat\beta)$
        shows that $\Omega$ admits an inverse Jacobi multiplier $\beta\in K'$. This is a special case of a general result obtained by Zhang \cite{zhang2016} by a different approach. See also \cite{ACPW}, subsection 3.1 for three dimensional vector fields.
    \end{remark}
    
\subsection{The setting when $\mathcal I$ contains no closed form}
Throughout this subsection we assume that  $\mathcal I$ contains no closed form. 
We define a relation via
    \begin{equation}
        \omega_i\sim\omega_j:\Longleftrightarrow \alpha_i-\alpha_j=\dfrac{d\ell_{ij}}{\ell_{ij}},\quad\text{some  }\ell_{ij}\in L^*.
    \end{equation}
    This is obviously an equivalence relation, and it is respected by the action of the Galois group; thus
    \[
    \omega_i\sim\omega_j\Longleftrightarrow  \sigma(\omega_i)\sim\sigma(\omega_j),\quad\text{all }\sigma\in G.
    \]
Moreover we set
    \begin{equation}
        {\mathcal B}_1:=\left\{\omega_j;\, \omega_j\sim\omega_1\right\},
    \end{equation}
    then (possibly by renumbering) we may assume that
    \begin{equation}
         {\mathcal B}_1=\left\{\omega_1, \ldots \omega_q\right\}
    \end{equation}
    and that $\omega_1,\ldots,\omega_s$ form a maximal subset linearly independent over $L$.\\
    We further define {\em  blocks} $\mathcal B_k$, $1\leq k\leq r$ as the images of $\mathcal B_1$ under the action of $G$, and set
    \[
    V_k:=\left<\mathcal B_k\right>_L.
    \]
    Note that $\dim_L \,V_k=s$ for all $k$.
    For later use we relabel
    \begin{equation}\label{blockrefine}
        \mathcal B_k=\left\{\omega_{k,1},\ldots,\omega_{k,s}\right\}\cup \left\{ \omega_{k,s+1},\ldots,\omega_{k,q}\right\},
    \end{equation}
    with $\omega_{k,1},\ldots,\omega_{k,s}$ forming a maximal $L$-linearly independent set\footnote{In the case $q=s$, the second set is empty.}. With relabeling we also write
    \[
    d\omega_{k,i}=\alpha_{k.i}\wedge\omega_{k,i},\quad d\alpha_{k,i}=0.
    \]
We obtain a decomposition of $V$:
    \begin{lemma}\label{dirsumlem}
        $V$ is the direct sum of $V_1,\ldots,V_r$, and $rs\leq n-1$.
    \end{lemma}
    \begin{proof} The second assertion is immediate from the first. The first assertion amounts to showing that the set $\{\omega_{k,1},\ldots,\omega_{k,s},\,1\leq k\leq r\}$ is linearly independent.  Assume linear dependence, and choose a minimal linearly dependent subset, with elements relabeled as $\omega_1^*,\ldots,\omega_m^*$ (with $d\omega_k^*=\alpha_k^*\wedge\omega_k^*$). Thus there exist $a_k\in L$, all nonzero, so that 
        \[
        \sum_{k=1}^m a_{k}\,\omega_{k}^*=0.
        \]
        We may choose $a_1=-1$, hence
        \[
        \omega_1^*=\sum_{k=2}^m a_{k}\,\omega_{k}^*
        \]
        with linearly independent $\omega_2^*,\ldots,\omega_m^*$.
        Differentiation yields
        \[
        \begin{array}{rcl}
           \alpha_1^*\wedge\omega_1^*=  d\omega_1^*& =& \sum_{k}da_{k}\wedge\omega_{k}^*+\sum_{k}a_{k}\,\alpha_{k}^*\wedge\omega_{k}^* \\   &=&\sum_{k=1}^m\left(da_{k}+a_{k}\alpha_{k}^*\right)\wedge\omega_{k}^*.\\
        \end{array}
        \]
        Combining this with the expression for $\omega_1^*$, we find that
        \begin{equation}\label{trick}
        \sum_{k=2}^m\left(da_k+a_k\alpha_k^*-a_k\alpha_1^*\right)\wedge\omega_k^*=0.
        \end{equation}
        Now assume that some $\beta_j:=da_j+a_j\alpha_j^*-a_j\alpha_1^*\not=0$. Forming the wedge product with all the $\omega_k^*$, $k\not=j,\,1$ one has
        \[
        \beta_j\wedge\omega_2^*\wedge\cdots\wedge\omega_m^*=0,
        \]
        while Lemma \ref{auxlem} shows that there exist $\eta_1,\ldots,\eta_{n-m+1}\in K'$ such that 
        \[
\omega_2^*\wedge\cdots\wedge\omega_m^*\wedge\eta_1\wedge\cdots\wedge\eta_{n-m+1}=\ell\Omega
        \]
        with some $\ell\not=0$. 
        With $\widehat\beta_j:=\frac{1}{a_j}\beta_j$ one finds
        $\widehat\beta_j\wedge\Omega=0$, $d\widehat\beta_j=0$, which contradicts our assumption that $\mathcal I$ contains no closed form. Therefore all $\beta_k=0$, and 
        \[
        \alpha_k^*-\alpha_1^*= -\dfrac{da_k}{a_k}; \quad\text{hence}\quad \alpha_k^*\sim\alpha_1^*.
        \]
        By the definition of blocks, all $\omega_k^*$ lie in the same block, say $\mathcal B_i$. But this contradicts the linear independence of $\omega_{i,1},\ldots,\omega_{i,s}$.
        \end{proof}
We now can dispose of the case $s=1$, $r>1$ (for the case $s=r=1$ recall Proposition \ref{boreprop}): 
\begin{proposition}\label{permgroup} 
Assume that $\dim_L\,V_k=1$ for all $k$. Then $\mathcal{B}_k=\{\omega_{k,1}\}$, thus $\left|\mathcal{B}_k\right|=1$ for all $k$, and $\omega_{1,1},\ldots,\omega_{r,1}$ form a basis of $V$, on which $G$ acts as a permutation group, with  $[L:K]=\left|Gal(L:K)\right|=r$.
\end{proposition}
\begin{proof}
We have  $q=s=1$ from Lemma \ref{G0simplem}. The remaining assertions are then obvious.
\end{proof}
         From here on we will assume that $s>1$. We introduce a further stabilizer subgroup.

     \begin{definition}
         Given that $s>1$, let
         \[
         G_1:=\left\{\sigma\in G;\,\sigma(\mathcal B_k)=\mathcal B_k,\quad 1\leq k\leq r\right\},
         \]
and denote by $M$ the fixed field of $G_1$.
\end{definition}
         Thus $G_1$ stabilizes all subspaces $V_k$, and obviously $G_1$ is normal in $G$. \\
We next construct a different basis for $V_1$: Possibly renumbering variables, with Gauss-Jordan one obtains
         \[
V_1=\left<dx_1+\sum_{j=1}^{n-s} k_{1,j}\,dx_{j+s},\ldots,dx_s+\sum_{j=1}^{n-s} k_{s,j}\,dx_{j+s}\right>
         \]
         over $L$. Averaging over $G_1$ yields
         \begin{equation}
             \mu_i:=dx_i+\dfrac{1}{|G_1|}\sum_{\sigma\in G_1}\sum_{j=1}^{n-s}\sigma(k_{i,j})dx_{j+s}\in V_1,\quad 1\leq i \leq s,
         \end{equation}
         and linear independence of the $\mu_i$ shows
         \[
         V_1=\left<\mu_1,\ldots,\mu_s\right>
         \]
         over $L$, with $\tau(\mu_i)=\mu_i$ for all $\tau\in G_1$, $1\leq i\leq s$.
         Setting
         \[
         \widetilde\Omega:=\mu_1\wedge\ldots\wedge\mu_s
         \] 
         we see from above that $\widetilde\Omega=k\cdot\omega_{1,1}\wedge\cdots\wedge\omega_{1,s} $ with some $k\in L^*$. 
  \begin{lemma}\label{lem15one}
           There exists $\ell\in L$ such that the form   $\widetilde\omega_{1,1}:=\widetilde \ell^{1/s}\,\omega_{1,1}$ satisfies
           \begin{equation}
               d\widetilde\omega_{1,1}=\widetilde\alpha_{1}\wedge\widetilde\omega_{1,1}; \quad\widetilde\alpha_1:=\alpha_{1,1}+\dfrac{1}{s}\dfrac{d\widetilde \ell}{\widetilde\ell}\in M'.
           \end{equation}
         \end{lemma}
\begin{proof} A straightforward computation shows 
         \[
         d\widetilde\Omega=\rho\wedge\widetilde\Omega; \quad \rho:=\dfrac{dk}{k}+\alpha_{1,1}+\cdots+\alpha_{1,s}.
         \]
         By construction, 
         $\sigma(\widetilde\Omega)=\widetilde\Omega$ for all $\sigma\in G_1$, so $d\widetilde\Omega=\sigma(\rho)\wedge\widetilde\Omega$ and
         \[
\left(\sigma(\rho)-\rho\right)\wedge\widetilde\Omega=0,\quad\text{all}\quad \sigma\in G_1.
         \] 
                 Now use Lemma {\ref{auxlem}} to extend $\omega_{1,1},\ldots,\omega_{1,s}$ by $\eta_{s+1},\ldots,\eta_{n-1}\in K'$ to a basis of $\widetilde V$, and to obtain
         \[
         \left(\sigma(\rho)-\rho\right)\wedge\Omega=0,\quad d\left(\sigma(\rho)-\rho\right)=0
         \]
         for all $\sigma\in G_1$. Since $\mathcal I$ contains no closed form, we see that $\rho$ is invariant by $G_1$. By equivalence we have $\alpha_{1,j}=\alpha_{1,1}+ \dfrac{d\widetilde\ell_{1,j}}{\widetilde\ell_{1,j}}$ with some $\widetilde\ell_{1,j}\in L$, so 
         \[
        \rho=s\alpha_{1,1}+\dfrac{d\widetilde \ell}{\widetilde \ell},\quad\text{some}\quad\widetilde\ell\in L.
         \]
\end{proof}
         Now let $\widetilde L$ be the smallest Galois extension of $K$ that contains both $L$ and $\lambda:=\widetilde\ell^{1/s}$. Every $\tau\in Gal(L:K)$ extends to some $\widetilde\tau\in \widetilde G:=Gal(\widetilde L:K)$, with
        \[
\widetilde\tau(\lambda)^s=\widetilde\tau(\lambda^s)=\widetilde \tau(\widetilde\ell)=\tau(\widetilde\ell),
        \]
        so $\widetilde\tau(\lambda)$ is an $s^{\rm th}$ root of $\widetilde\ell$. Therefore $\widetilde\tau(\widetilde\omega_{1,1})\in \widetilde L^*\,\omega_{i,1}$ whenever $\tau(\omega_{1,1})=\omega_{i,1}$.\footnote{Upon passing from $G$ to $\widetilde G$ we can no longer argue that $\widetilde\sigma(\widetilde\omega)=\widetilde\omega$ when these forms are linearly dependent over $\widetilde L$.} \\
        From $\widetilde G$ one obtains extended blocks
        \[
        \widetilde{\mathcal B}_k=\left\{ \widetilde\omega_{k,1},\ldots,\widetilde\omega_{k,q}\right\}\supseteq \mathcal B_k,\quad 1\leq k\leq r,
        \]
        possibly with different $q$. As before, the $\widetilde{\mathcal B}_k$ are the images of $\widetilde{\mathcal B}_1$ under the action of $\widetilde G$. We let
        \[
        \widetilde G_1=\left\{ \widetilde\sigma;\,\widetilde\sigma(\widetilde{\mathcal B}_k)=\widetilde{\mathcal B}_k;\quad\text{all}\quad k\right\},
        \]
        which is a normal subgroup of $\widetilde G$. We denote by $\widetilde M$ the fixed field of $\widetilde G_1$, noting $M\subseteq\widetilde M$.
        \begin{lemma}\label{lem15two}
            For all $k$, $1\leq k\leq r$, there exist $\widetilde\alpha_k\in \widetilde M'$ such that, for all $\widetilde\omega_{k,j}\in \widetilde{\mathcal B}_k$,
            \begin{equation}
                d\widetilde \omega_{k,j}=\widetilde\alpha_k\wedge\widetilde\omega_{k,j}.
            \end{equation}
        \end{lemma}
        \begin{proof}For $\widetilde\omega_{1,1}$ this was shown in the previous lemma. The remaining assertions follow by the action of $\widetilde G$.
        \end{proof}
\begin{remark}{\em
Mutatis mutandis, the statement of Lemma \ref{lem15two} also holds, obviously, when $\mathcal I$ contains a closed form $\omega\in L'$.  In this case we have one block $\mathcal B_1$ consisting of all images of $\omega$ under the action of the Galois group, and $\widetilde\alpha_1=0$. (For the case $s=1$ see, again, Proposition \ref{boreprop}.)
}
\end{remark}
\subsection{The general setting}
We continue to assume that $s>1$, thus each $V_k$ has dimension $>1$ over $L$.
As a matter of notation we drop all tildes from here on (in other words, we assume that we started with $\widetilde L$ right away).\\
 In case $q=s$ every matrix representing $\sigma\in G$ is a permutation matrix. Now assume $q>s$.
        Using the notation of \eqref{blockrefine}, for fixed $k$ we have that 
        \[
        \omega_{k,s+j}=\sum_{i=1}^s a_{ij}\omega_{k,i},\quad 1\leq j\leq q-s,
        \]
         with $a_{ij}\in L$. But furthermore we see
         \[
         \sum_ia_{ij}\,\alpha_k\wedge\omega_{k,i} =\alpha_k\wedge\omega_{k,s+j}=d\omega_{k,s+j}=\sum_i\left(da_{ij}+a_{ij}\alpha_k\right)\wedge\omega_{k,i},
         \]
         hence 
         \[
         \sum_{i=1}^s da_{ij}\wedge\omega_{k,i}=0.
         \]
         Now fix $i^*$, take the wedge product with all $\omega_{k,p}$, $p\not= i^*$ and use Lemma \ref{auxlem} to see that 
         \[
        d a_{i^*, j}\wedge\Omega=0.
         \]
         So, by Lemma \ref{ratfi} we obtain the following result:
         \begin{lemma}\label{allinClem}
             If $\Omega$ admits no rational first integral, then all $a_{ij}\in\mathbb C$.
         \end{lemma}
         We summarize:
         \begin{proposition}\label{summaryprop}
             Assume that $s>1$ and that $\Omega$ admits no rational first integral (and $L=\widetilde L$ w.l.o.g.). Then for every $\sigma\in G$ the matrix $A(\sigma)\in GL(rs,L)$ that represents its action on the basis $\omega_{1,1},\ldots,\omega_{1,s},\ldots,\omega_{r,1},\ldots,\omega_{r,s}$ actually has all entries in $\mathbb C$.
         \end{proposition}
         \begin{proof}
             \begin{enumerate}[(i)]
             \item For $\sigma\in G_1$ and fixed $k$, one sees with Lemma \ref{allinClem} that $\sigma$ sends every $\omega_{k,j}$ to a $\mathbb C$-linear combination of the $\omega_{k,i}$.
             \item Now let $\tau\in G\setminus G_1$, and $k$ fixed. Then we have 
             \[
        \tau(\omega_{k,i})=\omega_{\ell,j}, \quad \text{ for some } \ell\leq r,\ \text{and  some } j\leq q,
             \]
             and by the first part, $\omega_{\ell,j}$ is a $\mathbb C$-linear combination of $\omega_{\ell, 1},\ldots,\omega_{\ell, s}$.
             \end{enumerate}
         \end{proof}
          Thus we have a linear representation of the Galois group $G$ of $L$ in some $GL(rs,\mathbb C)$, with $rs\leq n-1$.  Note the block form of the matrices, corresponding to the blocks $\mathcal B_k$.\\

\begin{lemma}\label{domega0lemplus}
    Given the setting of Proposition \ref{summaryprop}, there is an intermediate field $F$ of the extension $L\supseteq K$ such that $F$ is Galois over $K$, $\omega\in F'$ and the representation of $Gal\,(F:K)$ in $GL(rs,\,\mathbb C)$ is faithful. Thus we may assume that the Galois group is isomorphic to a finite subgroup of $GL(rs,\,\mathbb C)$.
\end{lemma}
\begin{proof}
     Let $\widehat G$ be the image of $\tau\mapsto A(\tau)$, and $N\subseteq G$ its kernel; moreover let $F$ the fixed field of $N$. Since $N$ is normal, $F$ is Galois over $K$ with Galois group isomorphic to $G/N$. By definition, the action of $N$ fixes every element of $V$; this shows $V\subseteq F'$. Moreover $G/N$ acts in a well-defined manner on forms in $V$ via $\tau N(\omega):=\tau(\omega)$, and this yields a canonical isomorphism from $G/N$ to $\widehat G$.
\end{proof}
         So in the case $s>1$ we arrive at finite groups of complex $rs\times rs$ matrices. By the solution of the inverse problem of differential Galois theory, every such group is the Galois group of some Picard-Vessiot extension of $K$. We will utilize this fact for three dimensional vector fields later on.
   
\subsection{Specialization to dimension three}
In this subsection we will consider three-dimensional rational vector fields
\begin{equation}\label{vfield3d}
\mathcal{X}=P\frac{\partial }{\partial x}+Q\frac{\partial }{\partial y}+R\frac{\partial }{\partial z};
\end{equation}
or the corresponding 2-forms
\begin{equation}\label{2form3d}
\Omega=P\,dy \wedge dz+Q\,dz \wedge dx+R\,dx \wedge dy
\end{equation}
defined over $K:=\mathbb C(x,y,z)$.

According to Lemma \ref{dirsumlem} and Proposition \ref{boreprop}, in dimension $n=3$, the exceptional cases satisfy $rs=2$.
\begin{proposition}\label{summary3prop}
    Let $(\Omega, L)$ represent an exceptional case in dimension three.
\begin{itemize}
\item[(E1)] In case $s=1$ and $r=2$ there exists a degree two intermediate field $F$ with $[F:K]=2$, and $\omega\in F'$, $\alpha\in F'\setminus K'$, satisfying condition  \eqref{probsetup}.
\item[(E2)] In case $s=2$ and $r=1$  there exists a cyclic extension $\widetilde L$ of $L$ which is Galois over $K$, such that \eqref{probsetup} holds for some $\widetilde\omega\in \widetilde L'$ and $\widetilde\alpha \in K'$, and the action of the Galois group $\widetilde G=Gal\,(\widetilde L:K)$ on $\widetilde \omega$ induces a faithful representation of $\widetilde G$ in $GL(2,\mathbb C)$.
\end{itemize}
\end{proposition}
\begin{proof} See Proposition \ref{permgroup} and Proposition \ref{summaryprop}. The case $s=1$, $\alpha\in K'$ is regular by Lemma \ref{purex}.
\end{proof}
Thus, in dimension three we have two possible exceptional types.

\section{Exceptional cases in dimension three}
The description of case (E1) is quite straightforward: We have a quadratic extension $L$ of $K$, and forms $\omega=\omega_1\in L'\setminus \{0\}$, $\alpha=\alpha_1\in L'\setminus K'$ such that \eqref{probsetup} is satisfied. The nontrivial element of the Galois group sends $\omega_1$ to $\omega_2$ and $\alpha_1$ to $\alpha_2$. We turn to (E2).

\subsection{Type (E2): Finite subgroups of $GL(2,\mathbb C)$}

The Galois group in case (E2) admits a faithful representation in $GL(2,\mathbb C)$. The finite subgroups of $GL(2,\mathbb C)$ were classified, up to conjugation, in Nguyen et al.\ \cite{NPT}, and the classification of finite subgroups of $SL(2,\mathbb C)$ is known (see e.g.\cite{NPT}, Thm.\ 3, or van der Put and Springer \cite{vdPS}, Thm.\ 4.29). These classifications are based on the well-known classification of the finite subgroups of 
$PGL(2,\mathbb C)$, viz.
\begin{itemize}
    \item the cyclic groups;
    \item the dihedral groups $D_n$;
    \item the alternating group $A_4$, the symmetric group $S_4$, and the alternating group $A_5$ (the symmetry groups of the Platonic solids).
\end{itemize}
Thus every finite subgroup of $GL(2,\mathbb C)$ has one of the above as image under the canonical projection to the projective linear group. 

As a preliminary step we show that it suffices to consider a representation of the Galois group in $SL(2,\mathbb C)$; possibly at the expense of a further degree two field extension.

\begin{lemma}\label{sl2lem}
  Given case (E2) and the setting of Proposition \ref{summary3prop}, the following hold.
    \begin{enumerate}[(a)]
        \item There is an extension $\widehat L=L(k)$ of $L$, of degree $\leq 2$, which is Galois over $K$, with $k^2\in L$.
        \item For every $\widehat\sigma\in \widehat G:= Gal(\widehat L:K)$ that extends $\sigma\in G=Gal(L:K)$, there exists $\ell_{\widehat\sigma}\in \mathbb C$ such that 
        \[
        \widehat\sigma(k)=\ell_{\widehat\sigma}\cdot k;\quad \ell_{\widehat\sigma}^2=\det A(\sigma).
        \]
        \item With $\widehat\omega_i:=\frac{1}{k}\ \omega_i$, $1\leq i\leq r$ one has
        \[
        d\widehat\omega_i=\widehat\alpha\wedge\widehat\omega_i,\quad \widehat\alpha\in K',\, d\widehat\alpha=0.
        \]
        \item Moreover 
        \[
        \widehat\sigma (\begin{pmatrix}\widehat\omega_1\\ \widehat\omega_2\end{pmatrix}) =B(\widehat\sigma)\, \begin{pmatrix}\widehat\omega_1\\ \widehat\omega_2\end{pmatrix}
        \]
        with
        \[
        B(\widehat\sigma)=\frac{1}{\ell_{\widehat\sigma}}\, A(\sigma)\in SL(2,\mathbb C).
        \]
        In particular, the projections of $B(\widehat\sigma)$ and $A(\sigma)$ to $PGL(2,\mathbb C)$ are equal.
    \end{enumerate}
\end{lemma}
\begin{proof}
    For parts (a) and (b), let $\eta,\theta\in K'$ be linearly independent such that $\eta\wedge\Omega=\theta\wedge\Omega=0$. Then there exists $C\in GL(2,L)$ such that 
    \[
    \begin{pmatrix}
        \omega_1\\ \omega_2
    \end{pmatrix}=C\cdot \begin{pmatrix}
        \eta\\ \theta
    \end{pmatrix}.
    \]
    Apply $\sigma\in G$ to see
    \[
    \sigma(C)\cdot\begin{pmatrix}
        \eta\\ \theta
    \end{pmatrix}=\sigma (\begin{pmatrix}
        \omega_1\\ \omega_2
    \end{pmatrix})=A(\sigma)\cdot\begin{pmatrix}
        \omega_1\\ \omega_2
    \end{pmatrix}=A(\sigma)\cdot C\cdot \begin{pmatrix}
        \eta\\ \theta
    \end{pmatrix},
    \]
    hence $\sigma(C)=A(\sigma)\cdot C$, and $\Delta:=\det C\in L$ satisfies $\sigma(\Delta)=\det A(\sigma)\cdot \Delta$.\\
    Now let $k$ such that $k^2=\Delta$, and  $M=L(k)$ (so possibly $M=L$ when $k\in L$). When $M\not=L$, then for given $\sigma\in G$ define $\widehat \sigma\in Gal(M:K)$, by choosing $\ell_{\widehat\sigma}\in\mathbb C$ with $\ell_{\widehat\sigma}^2=\det A(\sigma)$, and $\widehat\sigma|_{L}=\sigma$, $\widehat\sigma(k)=\ell_{\widehat\sigma}\cdot k$. (There are two such extensions.) This is consistent with
    \[
    \widehat\sigma(k)^2=\widehat\sigma(k^2)=\sigma(k^2)=\det A(\sigma)\cdot k^2,
    \]
and one verifies that $M$ is Galois over $K$.
\\
    As to part (c), let $\widehat\sigma\in\widehat G  $, thus $\widehat\sigma(k)=\ell_{\widehat\sigma}\cdot k$. Then $\widehat\sigma(dk)=\ell_{\widehat\sigma}\cdot dk$, and 
    \[
    \widehat\sigma\left(\frac{dk}{k}\right)=\frac{dk}{k},\quad
    \text{hence  } \frac{dk}{k}\in K'.
    \]
    Now 
    \[
    d(\widehat \omega_i)= d\left(\frac{1}{k}\,\omega_i\right)=\left(-\frac{dk}{k^2}+\frac{1}{k}\alpha\right)\wedge\omega_i =\left(-\frac{dk}{k}+\alpha\right)\wedge\widehat\omega_i.
    \]
    As to part (d), we have
    \[
    \widehat\sigma\left(\frac{1}{k}\,\omega_i\right)=\frac{1}{\ell_{\widehat\sigma}}\,\frac{1}{k}\,\sigma(\omega_i),
    \]
    which shows the first assertion about $B(\widehat\sigma)$, and moreover
    \[
    \det B(\widehat\sigma)=\ell_{\widehat\sigma}^2\,\det A(\sigma)=1.
    \]
    Since $B(\widehat\sigma)$ is a constant multiple of $A(\sigma)$, it projects to the same element of $PGL(2,\mathbb C)$.
\end{proof}
From here on we consider an exceptional case (E2) with the Galois group $G=Gal(L:K)$ admitting a faithful representation in $SL(2,\mathbb C)$. 

According to \cite{vdPS}, Theorem 4.29, the finite subgroups of $SL(2,\mathbb C)$ may be grouped in three classes, up to conjugacy:
\begin{enumerate}
    \item Finite subgroups of the Borel subgroup 
    \[
    \left\{\begin{pmatrix}
        a& b\\ 0&a^{-1}
    \end{pmatrix};\quad a\in\mathbb C^*,\,b\in\mathbb C\right\}.
    \]
    Since every such matrix with $b\not=0$ has infinite order, a finite subgroup is diagonal, hence cyclic, generated by a diagonal matrix with a root of unity $a$.
    \item Finite subgroups of the infinite dihedral group 
    \[
    D_\infty:=\left\{\begin{pmatrix}
        a&0\\0&a^{-1}
    \end{pmatrix}, a\in\mathbb C^*\right\}\bigcup \left\{\begin{pmatrix}
        0&b\\-b^{-1}&0
    \end{pmatrix}, b\in\mathbb C^*\right\}.
    \]
    Given a finite $H\subseteq D_\infty$, the subgroup of diagonal matrices is cyclic, and $H$ is generated by this subgroup and any non-diagonal element, which we take as $\begin{pmatrix}
        0&1\\ -1&0
    \end{pmatrix}$.
    \item One of the groups $A_4^{SL(2)}$, $S_4^{SL(2)}$ or $A_5^{SL(2)}$, the respective inverse images of the symmetry groups of the Platonic solids. 
\end{enumerate}
\begin{remark}\label{nocevec}   {\em  We note for later use: If any of these finite groups admits a common eigenvector then it is cyclic.}
\end{remark}
We proceed case-by-case. Case 1 is taken care of by Lemma \ref{purex}, as follows:
\begin{proposition}
  Given exceptional case (E2), the Galois group $G$ cannot be cyclic.
\end{proposition}
Case 2 can be reduced to quadratic extensions of $K$:
\begin{proposition}\label{dihedprop}
    Let $(\Omega,\,L)$ represent an exceptional case (E2) with the representation of the Galois group $H$ {isomorphic to $D_N^{SL(2)}$, with $D_N$} a dihedral group. Then the exceptional case (E1) holds with some intermediate field $F$ of $L\supseteq K$.
\end{proposition}
\begin{proof}{
        We may assume that the representation $H$ is generated by
        \[
        R=\begin{pmatrix}
            \zeta&0\\0&\zeta^{-1}
        \end{pmatrix} \quad\text{and}\quad S=\begin{pmatrix}
            0&1\\-1&0
        \end{pmatrix}
        \]
        with $\zeta$ a primitive $n^{\rm th}$ root of unity. The cyclic subgroups $\left<R\right>$ and $\left<S\right>$ have trivial intersection when $n$ is odd, and intersection $\left\{I,\,-I\right\}$ when $n$ is even. The subgroup $H_0$  that
           is generated by $R$ and $S^2=-I$, has index two. For even $n$, $H_0=\left<R\right>$ is cyclic; for odd $n$ one has $H_0=\left<R\right>\cup\left\{I,\,-I\right\}\cong \mathbb Z_n\times\mathbb Z_2\cong\mathbb Z_{2n}$, which is also cyclic.\\
       The corresponding subgroup $G_0$ of the Galois group $G$ is cyclic, of index two. Its fixed field $F$ is therefore Galois over $K$, of degree two. By Lemma \ref{purex} there exists a form $\widehat \omega\in \mathcal I\cap F'$; hence $(\Omega, F)$ represents exceptional case (E1).}
\end{proof}
For Case 3 we first ascertain that this will not automatically imply (E1).
\begin{proposition} Let $(\Omega, L)$ represent (E2) with $G=Gal(L:K)$ isomorphic to $A_4^{SL(2)}$ or $A_5^{SL(2)}$. Then $L$ contains no intermediate field  representing (E1).
\end{proposition}
\begin{proof} We observe that there exists no intermediate field $F$ of degree $2$ over $K$; equivalently, that $G$ contains no subgroup of index $2$. We include a proof of this fact for the sake of completeness.
\begin{itemize}
    \item $G=A_4^{SL(2)}$ has order $24$ and projects onto $A_4\subset PGL(2)$, with kernel $\{I,\,-I\}$. Assume that a subgroup $H$ of order $12$ exists. Then its image $\widetilde H$ under the projection would have order $6$ when $-I\in H$, and order $12$ otherwise. The first case would imply that $A_4$ admits a subgroup of order $6$, which is not the case. The second would imply that $A_4$ has a faithful representation in $GL(2)$, which is not the case (see e.g. Sagan \cite{Sagan}). 
    \item For $G=A_5^{SL(2)}$, of order $120$, an obvious variant of the above shows that there is no subgroup of order 60.
\end{itemize}
\end{proof}
\begin{remark}{\em  In the case $G=S_4^{SL(2)}$ one obtains, by analogous reasoning to the above, that the only index two subgroup is $A_4^{SL(2)}$. Beyond this observation, we leave this case open.}
\end{remark}

\subsection{Construction of a class of exceptional cases.}
So far, we have not proven that exceptional cases do exist.
In the present subsection we close this gap.
	
Let $K = \C(x,y,z)$ and $K_0 = \C(x)$. Unless specified otherwise,
``algebraic" always means algebraic over $K$ or $K_0$.
The derivative $h'$ will always refer to $\partial_x h$
and its use implies that $h$ has no dependence on $y$ or $z$.
	
 Let $\mathcal{G}\subseteq SL(2, \C)$ be isomorphic to one of the preimages of the Platonic symmetry groups\footnote{In this subsection we denote groups by calligraphic letters.}:
\begin{equation}\label{platolist}
A_4^{SL(2)}, \quad S_4^{SL(2)},\quad \text{or  }A_5^{SL(2)}, 
\end{equation}
and let
$q(x)$ be a rational function such that the equation
\begin{equation}\label{HGE} 
	y''(x) + q(x) y(x) = 0,
\end{equation}
has differential Galois group $\mathcal{G}$.  (We refer to van der Put and Singer \cite{vdPS} for the Galois theory of linear differential equations.) Let $f_1$ be a
fixed solution of \eqref{HGE} and let $f_i$ be
its images under $\mathcal G$. We let $L_0$ be the minimal Galois
(Picard-Vessiot) extension of $K_0$ which contains
the $f_i$.

We will assume that $q(x)$ is in the standard
form given in Matsuda \cite{Matsu}, pp.~13-14 and p.~18:
\begin{equation}\label{qshape} q(x) =\frac 14\left( \frac{A}{x^2} + \frac{B}{x(x-1)}+ \frac{C}{(x-1)^2}\right),
\end{equation}
with 
\[A = 1-1/m^2,\qquad C = 1-1/n^2,
\qquad A+B+C = 1-1/p^2,\] 
and
where $m$, $n$ and $p$ are some permutation of $(2,3,3)$ (for $A_4^{SL(2)}$), resp.\ $(2,3,4)$ (for $S_4^{SL(2)}$), resp.\ $(2,3,5)$ (for $A_5^{SL(2)}$). 

Denote by $\Omega$ the 2-form associated to the vector field
\begin{equation}\label{SYS}
   \dot x = 1, \qquad
   \dot y = q(x)\,z - 1, \qquad
   \dot z = -y.
\end{equation}
Let $L:=L_0(y,z)$ and note that $L$ is Galois over $K$ with Galois group isomorphic to $\mathcal G$. One verifies that the 1-forms, 
\begin{equation}\label{omegas}
\omega_i = d\big(y\,f_i(x) + z\,f_i'(x)\big) + f_i(x)\,dx \in L',\end{equation}
satisfy
\[\omega_i\wedge\Omega = 0,\qquad d\omega_i=0.\]
Writing $f=f_1$ and $g=f_2$, we have that $f'g-g'f$ is a constant (Wronskian condition).
Scaling by an appropriate element of $\C$, we can assume
\[ f\,g' - g\,f' = 1,\]
which gives
\begin{equation}\label{omwedge}
\omega_1\wedge\omega_2 = \Omega.\end{equation}
As before, the action of $\mathcal{G}$ on the $f_i$ and the $\omega_i$ 
can be represented as
\begin{equation}\label{Grep}
\sigma (\begin{pmatrix} f \\ g \end{pmatrix})
=A_\sigma \begin{pmatrix} f \\ g \end{pmatrix},\qquad
\sigma(\begin{pmatrix} \omega_1 \\ \omega_2 \end{pmatrix})
=A_\sigma \begin{pmatrix} \omega_1 \\ \omega_2 \end{pmatrix},\qquad
\end{equation}
where $\sigma \mapsto A_\sigma$ is a representation of $\mathcal G$
in $SL(2,\C)$.\\

With these notions and preliminaries we state:
\begin{proposition}\label{excprop}
{System \eqref{SYS} is Liouvillian integrable but} $\mathcal{I}\cap K'=\emptyset$.
\end{proposition}
\begin{proof}[Outline of proof]
We need to exclude the existence of 1-forms $0\not=\omega,\alpha \in K'$ 
	such that 
	\begin{equation}\label{COND} \omega\wedge\Omega = 0,\qquad d\omega=\alpha\wedge\omega,
	\qquad d\alpha=0.\end{equation}
To this end we will show:
\begin{enumerate}[(i)]
	\item The integral $F(x) = \int f(x)\,dx$ is not algebraic.
	\item If \eqref{COND} holds, 
	then there exist two independent rational first integrals
	of \eqref{SYS}.
	\item If \eqref{COND} holds, then the existence of the first integral, 
	\[ \int \omega_1 = y\,f(x) + z\,f'(x) + F(x),\]
	implies that $F(x)$ is algebraic, contradicting (i).
\end{enumerate}
The detailed steps will be carried out in the following subsections.
\end{proof}

\subsubsection{The integral ${F = \int f\,dx}$ is not algebraic.}

Let $F = \int f\,dx$ be an integral of $f(x)$ (any choice
of constant).  If $F$ is algebraic then
we can extend $L_0$ to a minimal Galois extension, $M_0$, of $K_0$ which includes $F$.  

Let $\tilde{\mathcal{G}}$ be the Galois
group of $M_0$.  Any automorphism {$\sigma \in \mathcal{G}$} on $L_0$ will extend to an automorphism {$\tilde{\sigma} \in \tilde{\mathcal{G}}$} of $M_0$. Conversely,
any $K_0$-automorphism of $M_0$ takes $L_0$ to itself and so the
action on $M_0$ restricts to an automorphism on $L_0$.
The derivation $'$ on $K_0$ extends uniquely to $M_0$
and is compatible with the action of $\tilde{\mathcal{G}}$ in the sense that $\sigma(x') = (\sigma(x))'$ for all $\sigma\in \tilde{\mathcal{G}}$ and $x \in M_0$.

Let $\sigma$ be chosen such that $\sigma(f) = a f + b g$ with $b\neq 0$.  This is possible by Remark \ref{nocevec}. Extend {$\sigma$} to an automorphism {$\sigma^*$} of $M_0$ and 
let {$G = (\sigma^*(F) - a F)/b \in M_0$}; then $G' = g$ and $G$ is
algebraic over $K_0$.  

Now consider the
differential equation
\begin{equation} \label{FEQ}
	Y'''(x) + q(x)\,Y'(x) = 0 .
\end{equation}
This has independent solutions $F$, $G$ and $1$.  
We can therefore assume that $M_0$ is the Picard-Vessiot
field associated to \eqref{FEQ}.

{For
$\tau \in \tilde{\mathcal{G}}$, clearly we must have
\[\tau(\begin{pmatrix} F \\ G \\ 1 \end{pmatrix})
=
\begin{pmatrix}
\tilde A_\tau	&   b_\tau \\
{ 0}	&  1 
\end{pmatrix}
\begin{pmatrix} F \\ G \\ 1 \end{pmatrix},
\]
for some $\tilde A_\tau \in GL_2(\C)$ and $ b_\tau \in \C^2$.
Differentiating, we find that 
\[\tau_{|L_0}(\begin{pmatrix} f \\ g \end{pmatrix})
= \tilde A_\tau
\begin{pmatrix} f \\ g \end{pmatrix}.
\]
So, with \eqref{Grep} we must have $\tilde A_\tau=A_{\tau_{|L_0}}$ for all $\tau \in \tilde{\mathcal{G}}$.
This gives a map from $\tilde{\mathcal{G}}$ onto 
 $\mathcal{G}$.
If $\sigma$ maps to the identity in $\mathcal{G}$, then $\tilde A_\sigma = I$
and we have must have $b_\sigma = 0$, since the matrix
\[
\begin{pmatrix}
	I	&  b_\sigma \\
	{0}	&  1 
\end{pmatrix}
\]
has finite order. Therefore the map is an isomorphism, and we conclude
\[
[M_0:K_0]=\left|\tilde{\mathcal G}\right| =\left|\mathcal G\right| =[L_0:K_0],
\]
thus $M_0=L_0$.

Moreover we can replace $F$ and $G$ by $F+k$ and $G+l$, with
$k, l \in \C$ so that $b_\sigma=0$ for all $\sigma$:\footnote{Clearly, adding a constant $k$ to $F$ does not alter the property of being
algebraic or not.}
Choose
\[
\begin{pmatrix} k \\ l \end{pmatrix} = c: = \frac1{|{\mathcal{G}}|} \sum_{\sigma\in {\mathcal{G}}} A_\sigma^{-1}b_\sigma.
\]
From the relation $\tau(\sigma(F)) = \tau\sigma(F)$, we deduce
\[
  A_{\tau\sigma} = A_\sigma A_\tau,\qquad
  b_{\tau\sigma} = A_\sigma b_\tau + b_\sigma.
\]
Hence, for any $\tau \in {\mathcal{G}}$, we have
\[
c = \frac1{|{\mathcal{G}}|} \sum_{\sigma\in {\mathcal{G}}}
A_{\tau\sigma}^{-1}b_{\tau\sigma}
  = \frac1{|{\mathcal{G}}|} \sum_{\sigma\in {\mathcal{G}}}
  A_\tau^{-1}A_\sigma^{-1}(A_\sigma b_\tau + b_\sigma)
= A_\tau^{-1}b_\tau + A_\tau^{-1}c,
\]
giving
\[ b_\tau + c = A_\tau c,\]
for all $\tau \in {\mathcal{G}}$.  Thus,
\[ \sigma(\begin{pmatrix} F \\ G \end{pmatrix}+c)
= A_\sigma \begin{pmatrix} F \\ G \end{pmatrix} + b_\sigma + c
= A_\sigma (\begin{pmatrix} F \\ G \end{pmatrix}+c).
 \]

Now let $\xi$ be given by
\[\xi(x) = 
\begin{vmatrix} 
	F & G \\
	f & g
\end{vmatrix},
\]
hence
\[
\xi''(x) = 
\begin{vmatrix} 
	f & g \\
	f' & g'
\end{vmatrix} +
\begin{vmatrix} 
	F & G \\
	f'' & g''
\end{vmatrix} = 1 - q(x)\, \xi.
\]
Moreover, for all $\sigma \in {\mathcal{G}}$, we have
\[
 ( \sigma( \xi))(x) = 
\sigma\left(\begin{vmatrix} 
   	F & G \\
   	f & g
   \end{vmatrix}\right) =
\left|\begin{pmatrix} 
	F & G \\
	f & g
\end{pmatrix}A_\sigma^T\right| = |A_\sigma^T|\,\xi(x) = \xi(x),
\]
and thus $\xi(x)$ is a rational function.

On the other hand, assume that a rational solution of
\[ \xi''(x) + q(x)\,\xi(x) = 1,\]
exists. Expanding $\xi$ as a Laurent series 
in $x$ about $0$, we get {with \eqref{qshape} that}
$\xi(x) = k x^r + \cdots$ for some $r \in \Z$ and $k\neq 0$.
Comparing coefficients of $x^{r-2}$ yields
\[ r(r-1) + \tfrac14 (1-1/m^2) = 0,\]
which yields a contradiction. Therefore
$F$ cannot be algebraic.
\subsubsection{Rational solutions of \eqref{COND} imply that two
	independent rational first integrals of \eqref{SYS} exist.}

Assume that \eqref{COND} holds.  We therefore have two algebraic
functions, $r$ and $s$, {such that
\[ \omega = r\,\omega_1 + s\,\omega_2 \in K'\]
with $\omega_i$ as in \eqref{omegas}.
This gives, taking differentials,
\[ r\alpha\wedge\omega_1+s\alpha\wedge\omega_2=\alpha\wedge\omega=d\omega = dr\wedge\omega_1 + ds\wedge\omega_2 ,\]
and hence,}
\begin{equation}\label{ONE} (dr - r\alpha)\wedge\omega_1 + (ds - s\alpha)\wedge\omega_2 = 0 .\end{equation}
If $s = 0$ then $\sigma(\omega) = \omega$ implies
\[  \sigma(r) \sigma(\omega_1) = r \omega_1,\]
for all $\sigma \in \mathcal{G}$. However, this would mean { that $\sigma(\omega_1)=c_\sigma\omega_1$ with some character $\sigma\mapsto c_\sigma$ of $\mathcal{G}$, and, with 
\[
\sigma(\omega_1)=(\cdots)dx+\sigma(f_1) dy+\sigma(f_1')dz,
\]
furthermore that $\sigma(f_1) = c_\sigma f_1$ for all $\sigma\in \mathcal{G}$.}
This is impossible by Remark \ref{nocevec}.\\
Therefore, we have $r$ and $s$ both non-zero.
From \eqref{ONE} we can deduce that
\[ \left(\frac{dr}r - \alpha\right)\wedge \Omega = 
    \left(\frac{ds}s - \alpha\right)\wedge \Omega = 0,\]
and hence
\[\left(\frac{dr}r - \frac{ds}s\right)\wedge \Omega = 0.\]
Thus $r/s$ is either a non-trivial first integral
of \eqref{SYS}, or $s = \lambda r$ for some $\lambda \in \C$.\\
In the latter case, we have
\begin{equation}\label{rscomb}
\omega = r (\omega_1 + \lambda \omega_2).
\end{equation}
As before, $\sigma(\omega) = \omega$ implies that
\[
r\begin{pmatrix} 1 & \lambda \end{pmatrix}
\begin{pmatrix} \omega_1 \\ \omega_2 \end{pmatrix} =
\sigma(r)\begin{pmatrix} 1 & \lambda \end{pmatrix} A_\sigma\begin{pmatrix} \omega_1 \\ \omega_2 \end{pmatrix},
\]
and hence $\begin{pmatrix} 1 & \lambda \end{pmatrix}$ is a
common left eigenvector of the $A_\sigma$, {by the linear independence of the $\omega_i$}.  Once again,
this is impossible by Remark \ref{nocevec}.

Hence, we are left in the situation where $r/s$ is
a non-trivial algebraic first integral.  From this,
we have a rational first integral by Lemma \ref{ratfi}.\\
{Let $\phi\in K$ be a rational first integral.} Since $d\phi$ is a
rational 1-form satisfying \eqref{COND} with $\alpha=0$,
we can start afresh with
\[  d\phi = r \omega_1 + s \omega_2,\]
for some new choice of $r$ and $s$, algebraic over $K$.
As above, we see that $r$ and $s$ are both non-zero.
Taking differentials, we get
\[dr \wedge \omega_1 + ds \wedge \omega_2 = 0,\]
and hence $r$ and $s$ satisfy
\[  dr\wedge\Omega = ds\wedge\Omega = 0.\]

{We show that one of $dr\wedge d\phi$ or $ds\wedge d\phi$ is non-zero:}
Suppose that \[d\phi\wedge dr = d\phi\wedge ds = 0;\]
then {we can write 
\begin{equation}\label{rseq}
dr = \tilde{r}d\phi;\quad ds = \tilde{s}\,d\phi
\end{equation}
 for some algebraic elements,
$\tilde{r}$ and $\tilde{s}$.}  Thus,
	\[ d\phi\wedge( \tilde{r}\,\omega_1 + \tilde{s}\,\omega_2) = 0,\]
{and hence
\[\tilde{r}\,\omega_1 + \tilde{s}\,\omega_2 = \ell\, d\phi,\]
thus
\[
(\tilde r-\ell r)\omega_1+(\tilde s-\ell s)\omega_2=0
\]
for some $\ell$, algebraic over $K$.}
This gives \[\tilde{r}-\ell\,r 
= \tilde{s} - \ell\,s = 0,\]
{due to the linear independence of the $\omega_i$ over $L_0$, and hence, $d(r/s) = 0$ from \eqref{rseq}, and $r/s\in \mathbb C$.}  However, we have
already seen from the argument concerning \eqref{rscomb} that this is not possible.

Without loss of generality, we take $dr\wedge d\phi \neq0$, so that $r$ is an algebraic first integral independent
of $\phi$.
We now deduce the existence of a second 
rational first integral, $\psi$, with $d\psi\wedge d\phi\neq0$, with a variant of the proof of Lemma \ref{ratfi}.
Let
\[ r^n + k_{n-1} r^{n-1} + \cdots + k_0 = 0\]
be the minimal equation for $r$ over $K$.  
Acting by $\Omega\wedge d(\cdot)$ on this equation we obtain
\[ \Omega\wedge dk_{n-1}\, r^{n-1} + \cdots + \Omega\wedge dk_0 = 0, \]
and hence we have $\Omega\wedge dk_i = 0$ for all $i$.
If we have $d\phi \wedge dk_j \neq 0$ for some $j$ then
we can take $\psi = k_j$.  If not, then we have
$d\phi \wedge dk_i = 0$ for all $i$.  However, in this case,
by acting on the
minimal equation by $d\phi\wedge d(\cdot)$, we get
\[  (nr^{n-1} + (n-1) k_{n-1} r^{n-2} + \cdots + k_1)\, d\phi\wedge dr = 0 ,\]
so that $nr^{n-1} + (n-1) k_{n-1} r^{n-2} + \cdots + k_1 = 0$,
contradicting  minimality.

\subsubsection{If \eqref{COND} holds, then ${F}$ must be algebraic.}
Assume that $\phi$ and $\psi$ are two independent rational first integrals of \eqref{HGE}, and let $p\in \mathbb C^3$ such that both are defned at $p$ and $d\phi\wedge d\psi\,(p)\not=0$. We may assume that $\phi(p)=\psi(p)=0$. Let $R$ and $S$, respectively be the numerators of $\phi$ and $\psi$, and let $R_0$ resp.\ $S_0$ be the irreducible factors of $R$ resp.\ $S$ that vanish at $p$; these define an algebraic variety $\mathcal V\subset \mathbb C^3$. Thus the local trajectory of \eqref{HGE} through $p$ is contained in $\mathcal V$.
\\
Since $\dot x=1$ in \eqref{HGE}, the trajectory admits a local parameterization by $x$ in the form $(x,\eta(x),\zeta(x))$, where
$\eta$ and $\zeta$ are analytic functions of $x$.\\
Now let $Q_1=Q_1(x,y)$ be the resultant of $R_0$ and $S_0$ with respect to $z$, and  let $Q_2=Q_2(x,z)$ be the resultant of $R_0$ and $S_0$ with respect to $y$. Since $Q_1(x,\eta(x))=0$ for all $x$ in a neighborhood of $x(p)$, one cannot have $Q_1$ independent of $y$, thus $Q_1(x,\eta(x))=0$ shows that $\eta$ is an algebraic function of $x$. By the same token, $Q_2(x,\zeta(x))=0$ shows that $\zeta$ is an algebraic function of $x$.\\
Choosing power series expansions of $F$ and $G$, we can embed $M_0$ into the field of convergent Laurent series, 
$M = \C\{x\}[x^{-1}]$, thus we also may assume $\eta(x)$ and $\zeta(x)$ in $M$.\\
Let the value of the first integral $\Xi(x,y,z) = y\,f(x) + z\,f'(x) + F(x)$ at $p$ be $k\in \mathbb C$.  Then $\Xi = k$ along the trajectory through $p$.  Thus 
\[ \Xi(x,\eta(x),\zeta(x)) = k = \eta(x)\,f(x) + \zeta(x)\,f'(x) + F(x),\]
in $M$ for all $x$ near $x(p)$, and hence $F(x)$ is algebraic over $\C(x)$ since
$\eta(x)$, $\zeta(x)$, $f(x)$ and $f'(x)$ are all algebraic
over $\C(x)$.

\subsection{Dihedral groups and exceptional cases (E1)}

The arguments in the previous subsection also apply to the construction of exceptional cases $(\Omega, L)$ with {Galois groups ${\mathcal G}\cong D_N^{SL(2)}$}, $N\geq 3$: Starting with $q$ defined as in \eqref{qshape}, but now with $(m,n,p)$ some permutation of $(2,2,N)$ (cf. Matsuda \cite{Matsu}), the construction yields a Picard-Vessiot extension with differential Galois group $\mathcal G$. The rest of the argument applies verbatim.\\
With Proposition \ref{dihedprop} one sees furthermore that $L$ contains a degree two subfield $F$ that represents an exceptional case (E1). Thus by an indirect argument the existence of this class of exceptional cases is ascertained.  \\
We also provide an explicit example.
\begin{example}{\em 
    We consider the $D_3$ case and determine an intermediate field of degree two.
    \begin{enumerate}
        \item We specialize the construction for $(m,n,p)=(2,2,3)$, thus \eqref{qshape}
        specializes to
        \begin{equation}\label{qshapespec} q(x) = \frac{3}{16x^2} - \frac{11}{72x(x-1)}+ \frac{3}{16(x-1)^2}.
\end{equation}
By Proposition \ref{excprop} the vector field obtained by specialization of \eqref{SYS} is exceptional. Explicitly the differential equation is given as
\begin{equation}\label{specialex}
\begin{array}{rcl}
  \dot x   & =&1 \\
    \dot y & =& \left(\dfrac{3}{16x^2} - \dfrac{11}{72x(x-1)}+ \dfrac{3}{16(x-1)^2}\right)z-1\\
    \dot z&=&-y.
\end{array}
\end{equation}
From general arguments we know that \eqref{specialex} admits a first integral defined over some degree two extension of $K$.
We proceed with the ingredients of the construction for the dihedral case:
A solution basis for the linear differential equation \eqref{HGE} is obtained using the {\sc{Maple}} symbolic computing environment, \cite{Maple}, and is given by
\[\begin{array}{rcl}
     f_1&:=&\left(x(x-1)\right)^{1/4}\cdot \left(\sqrt{x}+\sqrt{(x-1)}\right)^{1/3}; \\
     f_2&:=&\left(x(x-1)\right)^{1/4}/ \left(\sqrt{x}+\sqrt{(x-1)}\right)^{1/3}.
\end{array}
\]
\\
We note
\[
\dfrac{f_1'}{f_1}=h_1:=\dfrac{1}{12x(x-1)}\left(3\cdot(2x-1)+2\sqrt{x(x-1)}\right),
\]
thus $h_1$ lies in a degree two extension $F:=K(\sqrt{k})$ of the rational function field $K$, by adjoining a square root of $k:=x(x-1)$. 

Now, according to \eqref{omegas} the form 
\[
\begin{array}{rcl}
   \omega_1  & :=& d\big(y\,f_1(x) + z\,f_1'(x)\big) + f_1(x)\,dx\\
     & =&\left(f_1-qzf_1+yf_1'\right)\,dx+f_1\,dy+f_1'\,dz
\end{array}
\]
is a first integral of the 2-form, with $d\,\omega_1=0$. (The same holds for $\omega_2$, mutatis mutandis.)
\item To explicitly obtain forms over the degree two intermediate field $F=K(\sqrt{k})$ (with the nontrivial automorphism $\sigma$ sending $\sqrt{k}$ to $-\sqrt{k}$) we set
\[
\widehat\omega_1:=\dfrac{1}{f_1}\cdot\omega_1=\left(1-qz+yh_1\right)\,dx+dy+h_1\,dz\in F'
\]
and $\widehat\omega_2=\sigma(\widehat\omega_1)$. By construction we have 
\[
d\widehat\omega_i=\widehat\alpha_i\wedge\widehat\omega_i,
\]
with 
\[
\widehat\alpha_1=-\dfrac{f_1'}{f_1} \,dx=-h_1\,dx; \quad \widehat\alpha_2=-\sigma(h_1)\,dx
\]
Thus, by construction the $\widehat\omega_i\in F'$ are first integrals, and $\widehat\omega_1\wedge\widehat\omega_2\in F\cdot\Omega$. (One may also verify this by direct computation.)
    \end{enumerate}
    }
\end{example}


}


\end{document}